\documentclass[11pt]{amsart} 

\usepackage{setspace} \onehalfspace
\usepackage[active]{srcltx}
\usepackage[all]{xy} 
\usepackage{amsmath} 
\usepackage{amssymb}
\usepackage{color}
\usepackage{xcolor}
\usepackage{amsthm} \usepackage[pdfauthor={Matei Toma},
pdftitle={Positive currents on non-k\"ahlerian surfaces},
pdfkeywords={Douady space, coherent sheaves, complex space}, pdfview={Fit}]{hyperref}

\AtBeginDocument{%
   \def\MR#1{}
} %aceasta comanda ar trebui sa elimine numarul MR din referintele bibliografice

\newtheorem{Th}{Theorem}[section] 

\newtheorem{Lemma}[Th]{Lemma}
\newtheorem{Cor}[Th]{Corollary} 
\newtheorem{Prop}[Th]{Proposition}
\newtheorem{Conj}[Th]{Conjecture} 

\newtheorem{Def}[Th]{Definition} 
\newtheorem{Not}[Th]{Notation}
 
\newtheorem{Rem}[Th]{Remark}

\def\cC{\mathcal{C}} 
\def\cD{\mathcal D }
\def\cE{\mathcal E} 
\def\cF{\mathcal F}

\newcommand{\C}{\mathbb{C}} 
\newcommand{\D}{\mathbb{D}}
\newcommand{\bH}{\mathbb{H}}
\newcommand{\N}{\mathbb{N}}

\newcommand{\R}{\mathbb{R}}
\newcommand{\T}{\mathbb{T}}

\newcommand{\Z}{\mathbb{Z}}

\def\Aut{\mbox{Aut}}

\def\Coker{\mbox{Coker}}
 
\def\d{\mbox{d}}
 
 \def\Div{\mbox{Div}}

\def\Ker{\mbox{Ker}}

\def\Nef{\mbox{Nef}}

 \def\Psef{\mbox{Psef}}

\def\Supp{\mbox{Supp}}

\title{Positive currents on non-k\"ahlerian surfaces}

\author{Ionu\c t Chiose and Matei Toma}

\address{Ionu\c t Chiose, 
Institute of Mathematics of the Romanian Academy,
Bucharest, Romania}
\email{Ionut.Chiose@imar.ro}

\address{Matei Toma, 
Universit\'e de Lorraine, CNRS, IECL, F-54000 Nancy, France}
\email{Matei.Toma@univ-lorraine.fr}
 \date{\today}
\thanks{ AMS
  Classification (2020): 32J15; secondary: 32U40.}

\begin{document}

\begin{abstract}
We propose a classification of non-k\"ahlerian surfaces from a dynamical point of view and show how the known  non-k\"ahlerian surfaces fit into it. 
\end{abstract}
\maketitle

\noindent

\section{Introduction}

Since Kodaira's foundational work on the classification of compact complex surfaces, non-k\"ahlerian surfaces have been a subject of interest for many complex geometers. Beside the elliptic non-k\"ahlerian surfaces and the Hopf surfaces which were studied by Kodaira, two further series of examples appeared in the seventies: the Inoue surfaces \cite{Inoue} and the Kato surfaces \cite{Kato}. According to the Global Spherical Shell Conjecture \cite{NakamuraSugaku} these classes  should  exhaust all  non-k\"ahlerian compact complex surfaces up to bimeromorphic equivalence. Some recent progress towards a solution of this conjecture was achieved by Andrei Teleman in \cite{TelemanInventiones}, \cite{TelemanAnnals},  \cite{TelemanDonaldson}. His approach is to study a certain moduli space of stable rank two vector bundles on a given surface $X$ and deduce the existence of a compact analytic curve on $X$.  

In this paper we look at objects on $X$ of a different nature, namely at positive $\d$-exact currents. It is known by \cite{HarveyLawson} and \cite{LamariJMPA} that every non-k\"ahlerian surface admits non-trivial such currents. Extending our approach from \cite{ChiTo} we introduce an invariant $I(T)$ of a positive $\d$-exact current $T$ on a non-k\"ahlerian compact complex surface and investigate its behaviour for the known classes of surfaces. This analysis leads us to a rough classification of non-k\"ahlerian surfaces into {\it parabolic} and {\it hyperbolic} surfaces, see Definition \ref{def:hyperbolic-parabolic}. Note that the commonly used invariants such as the Kodaira dimension, the algebraic dimension or the K\"ahler rank do not adapt well to the historical partition of non-k\"ahlerian surfaces into elliptic, Hopf, Inoue and Kato surfaces, or to Kodaira's partition into classes. (An example is Kodaira's class $VII$ which was given a slightly restricted area in the monograph \cite{BHPV}.) We show that the results of Marco Brunella's papers \cite{BrunellaCylinder}, \cite{BrunellaGreen1}, \cite{BrunellaGreen2} fit perfectly into our classification. These papers were a source of motivation for our investigation and we therefore dedicate this work to the memory of Marco Brunella. 

%%%%%%%%%%%%%%%%%%%%%%%%%%%%%
%%%%%%%%%%%%%%%%%%%%%%%%%%%%%

\section{Preparations}\label{section:currents}

 \subsection{Positive pluriharmonic $(1,1)$-currents on non-k\"ahlerian surfaces}\label{subsection:ph}
 
In this section $X$ will always stand for a non-k\"ahlerian compact complex surface. It is known that  any compact complex  surface admits some Gauduchon metric, that is a hermitian metric whose associated K\"ahler form is $i\partial\bar\partial$-closed. We shall call such forms {\em Gauduchon forms} and we shall fix one Gauduchon form $\omega$ on $X$. We introduce the following definition following Lamari, \cite{LamariJMPA}.

\begin{Def}
A $(1,1)$-current on $X$ will be said to be {\em nef} if it is a weak limit of positive $i\partial\bar\partial$-closed $(1,1)$-forms on $X$ (or equivalently a weak limit of Gauduchon forms). 
\end{Def}

Nef currents are clearly positive and pluriharmonic, i.e. $i\partial\bar\partial$-closed. In the case of surfaces, extending the characterization of compact non-K\"ahler manifolds given by Harvey and Lawson in \cite{HarveyLawson}, Lamari shows that any non-k\"ahlerian surface admits some non-trivial  nef current which is $\d$-exact, \cite[Theorem 7.1]{LamariJMPA}. Since its evaluation on the Gauduchon form $\omega$ is positive, it follows that its Bott-Chern cohomology class is non-zero. Moreover, up to a positive multiplicative constant there is only one such class in $H^{1,1}_{BC}(X,\R)$.  In the sequel we shall denote by $\tau$ a smooth representative of such a class. We fix the class $\{\tau\}$ by requiring $\int_{X}\tau\wedge\omega=1$.

Note also that the intersection form on $H^{1,1}_{BC}(X,\R)$ is negative semi-definite with totally isotropic space spanned by the class of $\tau$, cf. e.g. \cite{LamariJMPA}.

 \begin{Prop}\label{prop:decomposition}
Let $T$ be a positive, $i\partial\bar\partial$-closed $(1,1)$-current on $X$. Then $T$ has a decomposition
\begin{equation}
T=\sum_jc_j[E_j]+T'
\end{equation}
where $c_j\geq 0$ are positive real numbers, $E_{j}$ are irreducible compact curves on $X$ and $T'$ is a nef current. 
%Moreover, if $X$ is non-Enoki, the 
%above decomposition is unique.
\end{Prop} 
 \begin{proof} 
 If $X$ is non-elliptic, then there are finitely many compact curves $E_j$ on $X$ and from Theorem 4.10 in \cite{Bassanelli} it follows that $T$ can be written 
\begin{equation}
T=\sum_jc_j[E_j]+T'
\end{equation}
where $\chi_{E_j}T'=0,\forall j$. 
If $X$ is elliptic, i.e., if there exists a non-constant map $\pi:X\to Y$ to a compact complex curve $Y$, denote by ${\mathcal C}$ the set of all compact complex curves in $X$. 
If $\omega$ is a fixed Gauduchon form on $X$, then there exists $c>0$ such that $\int_E\omega\geq c$, $\forall E\in {\mathcal C}$, see Remark \ref{rem:lower-bound}. 
Now if $n\in{\mathbb N}$, denote by $$\displaystyle{\mathcal C}_n=\left\{E\in {\mathcal C}\big\vert \chi_ET\geq\frac 1n[E]\right\}.$$ 
We claim that ${\mathcal C}_n$ is finite.
Indeed, we have $\displaystyle{T\geq \sum_{E\in {\mathcal C}_n}\chi_ET\geq \sum_{E\in {\mathcal C}_n}\frac 1n[E]}$ and therefore 
$$\int_X\omega\wedge T\geq \sum_{E\in {\mathcal C}_n}\frac 1n \int_E\omega\geq \frac 1n\cdot c \cdot{\rm card}\,\,{\mathcal C}_n.$$ 
Denote by $T_n$ the $d$-closed current $\displaystyle{\sum_{E\in {\mathcal C}_n}\chi_ET}$. Clearly ${\mathcal C}_n\subset {\mathcal C}_{n+1}$, and therefore $T_{n+1}\geq T_n$. It implies that the weak limit of $(T_n)_n$ is a current of the form $\displaystyle{\sum_jc_j[E_j]}$, where $c_j>0$ and $E_j$ are compact curves in $X$. It is a $d$-closed current, and $\displaystyle{T':=T-\sum_jc_j[E_j]}$ is a positive $i\partial\bar\partial$-closed current. 
From the construction of $\displaystyle{\sum_jc_j[E_j]}$, it follows that $\chi_ET'=0$, $\forall E\in {\mathcal C}$.

Therefore, on any non-K\"ahler compact surface, the positive $i\partial\bar\partial$-closed $(1,1)$-currents admit a Siu decomposition.

We have to prove that $T'$ is a nef current, i.e., that it belongs to $\overline{\mathcal G}$, the weak closure
of the cone of Gauduchon metrics ${\mathcal G}$ in ${\mathcal D}'^{1,1}(X,{\mathbb R})$ the space of $(1,1)$-forms with distribution coefficients.

Suppose that $T'\notin \overline{\mathcal G}$; then let $K=\{G\in \overline{\mathcal G}\vert \langle \omega, G\rangle=1\}$ where $\omega$ is our fixed Gauduchon form
 and $L={\mathbb R}T'\subset {\mathcal D}'^{1,1}(X,{\mathbb R})$. 
Since $L\cap K=\emptyset$, $K$ is weakly compact and $L$ is closed, they can 
be separated by a ${\mathcal C}^{\infty}$ $(1,1)$-form $\theta$ such that 
$\langle \theta , G\rangle \geq \varepsilon_0>0, \forall G\in K$ and 
$\langle \theta, G\rangle\leq 0,\forall G\in L$. We obtain 
$\langle \theta, T'\rangle =0$ and from Lemme 1.4 in \cite{LamariJMPA} that there exists $\varphi$ a distribution such that 
\begin{equation}
\theta+i\partial\bar\partial \varphi\geq\varepsilon_0\omega.
\end{equation}
It follows that $\varphi$ is actually quasi-plurisubhamonic, and from the regularization Theorem 3.2 in \cite{DemaillyPaun}, we can approximate $\varphi$ with another quasi-plurisubharmonic function 
$\varphi'$ which has logarithmic poles (in particular the set 
$E_+=\{x\in X\vert \nu(\varphi', x)>0\}$ is an analytic subset of $X$), and such that
\begin{equation}
i\partial\bar\partial \varphi'\geq\frac{\varepsilon_0}{2}\omega-\theta.
\end{equation} 
Apply Corollaire 3.2 in \cite{LamariJMPA} with $\alpha=0$, $Y=E_+$ and 
\begin{equation}
\gamma=\frac{\varepsilon_0}{2}\omega-\theta.
\end{equation}
 Since $\chi_{E_+}T'=0$, it follows that 
\begin{equation}
0=\langle 0, T'\rangle\geq \frac{\varepsilon_0}{2}\langle \omega,T'\rangle-\langle\theta, T'\rangle=\frac{\varepsilon_0}{2}\langle \omega, T'\rangle
\end{equation}
Hence $T'=0$, contradiction.
 \end{proof}
 
In the above proof we made use of the following 

\begin{Rem}\label{rem:lower-bound}  If $(X,\omega)$ is an $n$-dimensional compact complex manifold endowed with a Gauduchon metric, then there is a constant $c>0$ such that for any positive divisor $E$ on $X$ we have
$$ \int_E\omega^{n-1}\ge c.$$
This follows as in   \cite[p. 4]{TomaCriteria} from the fact that the volume function with respect to $\omega$ is pluriharmonic on the cycle space of codimension one cycles, \cite[Proposition 1]{BarletConvexitate}, combined with the fact that the set of all cycles whose volume is bounded from above  by some constant $M$ is compact. 
\end{Rem}
 
 \begin{Prop}
 \label{prop:pluriharmonic} Let $T$ be a positive $i\partial\bar\partial$-closed $(1,1)$-current such that $\int_X\tau\wedge T=0$. Then $T$ is closed. If, moreover, $T$ is nef, then it is $d$-exact.
\end{Prop}
 \begin{proof}
 Since $\int_X\tau\wedge T=0$, it follows that $\int_X\tau\wedge T'=0$, where $T'$ is the nef current that appears in the previous Proposition \ref{prop:decomposition}. Thus $T'$ is a weak limit of Gauduchon forms $T'=\lim \omega_n$ and each $\omega_n$ can be written
 \begin{equation}\label{gauduchon}
 \omega_n=\varepsilon_n \omega+\alpha_n+\partial\bar\sigma_n+\bar\partial\sigma_n,
 \end{equation}
 where 
\begin{equation}
\varepsilon_n=\int_X\tau\wedge \omega_n\to\int_X\tau\wedge T'=0,
 \end{equation}
 $\alpha_n$ are $d$-closed $(1,1)$-forms, and $\sigma_n$ are $(1,0)$-forms. Then
 \begin{equation}
 0\geq\int_X\alpha^2_n=\int_X(\alpha_n+d(\sigma_n+\bar\sigma_n))^2=\int_X(\omega_n-\varepsilon_n\omega+\partial\sigma_n+\bar\partial\bar\sigma_n)^2=
 \end{equation}
 \begin{equation*}
 =\int_X(\omega_n-\varepsilon_n\omega)^2+2\int_X\partial\sigma_n\wedge\bar\partial\bar\sigma_n=\int_X\omega_n^2-2\varepsilon_n\int_X\omega_n\wedge \omega+\varepsilon_n^2\int_X\omega^2+2\int_X\partial\sigma_n\wedge\bar\partial\bar\sigma_n\geq
 \end{equation*}
 \begin{equation*}
 \geq -2\varepsilon_n\int_X\omega_n\wedge \omega+\varepsilon_n^2\int_X\omega^2+2\int_X\partial\sigma_n\wedge\bar\partial\bar\sigma_n\geq
 \end{equation*}
 Since $\int_X \omega_n\wedge \omega\to\int_XT'\wedge \omega$ and $\varepsilon_n\to 0$, it follows that 
 \begin{equation}
 \int_X\partial\sigma_n\wedge\bar\partial\bar\sigma_n\to 0
 \end{equation}
 and therefore $\partial\sigma_n\to 0$ weakly. So from \eqref{gauduchon}
 \begin{equation}
 \partial T'=\lim \partial \omega_n=\lim (\varepsilon_n \partial \omega+\partial\bar\partial\sigma_n)=-\lim\bar\partial\partial\sigma_n=0,
 \end{equation}
 therefore $T'$ is closed and hence $T$ is closed as well.
 
 If $T$ is nef and closed, let $\alpha$ be a ${\mathcal C}^{\infty}$ representative of $T$ in the Bott-Chern cohomology class of $T$, i.e.,
$T=\alpha+i\partial\bar\partial\varphi$ where $\varphi$ is a quasi-plurisubharmonic function on $X$.
 If $T=\lim \omega_n$, where $\omega_n$ are Gauduchon forms, then
 \begin{equation}
 0\geq\int_X\alpha^2=\lim\int_X\alpha\wedge \omega_n=\lim\int_XT\wedge \omega_n\geq 0
 \end{equation}
 so $\int_X\alpha^2=0$ and $\alpha$ is $d$-exact and therefore $T$ is $d$-exact. We have used the fact that the intersection form on $H^{1,1}_{BC}(X,\R)$ is negative semi-definite with totally isotropic space spanned by the class of $\tau$. 
 \end{proof}
 
 \subsection{Positive exact  $(1,1)$-currents in $L^{2}_{-1}(X)$}\label{subsection:I}
We shall denote by $L^{2}(X)$ and by $L^{2}_{-1}(X)$ spaces of currents with coefficients in the corresponding spaces of functions without making their degrees precise. A closed positive current of bidegree $(1,1)$  is  in $L^{2}_{-1}(X)$ if it admits local $\partial\bar\partial$-potentials which are square integrable along with their gradients.
 
Bedford and Taylor defined in \cite{BedfordTaylor} the self intersection of 
 a closed positive $(1,1)$ current $T$ in $L^{2}_{-1}(X)$ 
  as follows: if $T=i\partial\bar\partial u$ on some open subset $U$ of $X$ and if $\psi $ is a test function on $U$ , then $$\int \psi T\wedge T=-\int i\partial\bar\partial\psi\wedge i \partial u\wedge\bar\partial u.$$
% see also \cite{Blocki}.
A direct computation shows that this definition does not depend on the chosen $i\partial\bar\partial$-potential $u$ and the definition is extended by linearity to define a current on $X$. By \cite[Theorem 3.6]{BedfordTaylor}  $T\wedge T$ is a positive $(2,2)$-current on $X$. This may also be seen in the following way. Let $\Omega$ be an open subset of $\C^{2}$. For a plurisubharmonic function $u$ in $L^{2}_{1}(\Omega)$ we define a distribution $MA(u)$ on $\Omega$ by setting
 \begin{equation}
 MA(u)(\psi):=-\int i\partial\bar\partial\psi\wedge i \partial u\wedge\bar\partial u.
 \end{equation}
We regularize $u$ in the usual way by means of a  sequence of regularizing kernels $(\rho_{\epsilon})_{\epsilon}$  converging to the Dirac distribution. The sequence of functions $u_{\epsilon}:=u\star \rho_{\epsilon}$ decreases towards $u$. The functions $u_{\epsilon}$  are in $\cC^{\infty}(\Omega)$ and plurisubharmonic on the smaller open sets $\Omega_{\epsilon}$. By the Meyers-Serrin theorem we also have $\lim_{\epsilon\to0} u_{\epsilon}=u$ in $L^{2}_{1}(\Omega)$. Thus if $\psi $ is a test function on $\Omega$, then $\Supp(\psi)\subset \Omega_{\epsilon} $ for $0<\epsilon<<1$ and $\lim_{\epsilon\to0} MA(u_{\epsilon})(\psi)=MA(u)(\psi)$ and on the other hand
$$
MA(u_{\epsilon})(\psi):=-\int_{\Omega} i\partial\bar\partial\psi\wedge i \partial u_{\epsilon}\wedge\bar\partial u_{\epsilon}=\int_{\Omega} \psi(i\partial\bar\partial u_{\epsilon})^{2}$$
which will be positive if $\psi$ is positive.

If $T$ is an exact positive $(1,1)$ current, % $T$ in $L^{2}_{-1}(X)$, 
then there exists  a bidegree $(0,1)$ current $S$ %in $L^2(X)$ 
such that $T=\partial S$. We investigate the situation when $T$  is in $L^{2}_{-1}(X)$. 

  \begin{Prop}\label{prop:S}
 Let $T$ be a positive $d$-exact current of bidegree $(1,1)$ in $L^{2}_{-1}(X)$ and let $T=\partial S$  for some  bidegree $(0,1)$-current $S$ in $L^2(X)$. Then $i\bar S\wedge S$ is $i\partial\bar\partial$-closed, 
and $\chi_Yi\bar S\wedge S=0$ for any compact analytic subset $Y$ of $X$. In particular, $i\bar S\wedge S$ is a nef pluriharmonic current. 

Moreover the value of the integral 
$$\int_{X}\tau\wedge i\bar S\wedge S$$
depends only on $T$ and not on the chosen primitive current $S$.
 \end{Prop}
 \begin{proof}
 Locally we may write $T=i\partial\bar\partial u$ and $S=i\bar\partial u$. It follows that: $\int \psi T\wedge T=-\int i\partial\bar\partial\psi\wedge i \bar S\wedge S$ for any $\cC^{\infty}$ function $\psi$ on $X$ and in particular estimating on $\psi=1$ one gets $T\wedge T=0$ and $i\partial\bar\partial(i\bar S\wedge S)=0$. 
 
 If $\dim Y=0$, the statement on the vanishing of $\chi_Yi\bar S\wedge S$ is well-known. If $\dim Y=1$, the statement follows from the fact that $i\bar S\wedge S$ has $L^1$ coefficients, and a $L^1$ function cannot dominate a Dirac measure. 
 
If $S_{1},$ $S_{2}$ are two primitive currents for $T$ as above, then $\eta:=\bar S_{1}-\bar S_{2}$ is a holomorphic $1$-form on $X$. If this form is non-zero then $i\eta\wedge\bar\eta$ is a non-trivial closed positive $(1,1)$-form such that $\int_{X}(i\eta\wedge\bar\eta)^{2}=0$ hence as remarked in Section \ref{subsection:ph}  $\{\tau\}=c\{i\eta\wedge\bar\eta\}\in H^{1,1}_{BC}(X,\R)$ for some positive constant $c$. Thus
$$\int_{X}\tau\wedge i\bar S_{1}\wedge S_{1}=c\int_{X}i\eta\wedge\bar\eta\wedge i(\bar S_{2}+\eta)\wedge (S_{2}+\bar\eta)=\int_{X}\tau\wedge i\bar S_{2}\wedge S_{2}. $$
 \end{proof}
 
 \begin{Not}\label{notation:I}
 Under the above  assumptions we shall use the following notation for the integral appearing in Proposition \ref{prop:S}
 $$I(T):=\int_{X}\tau\wedge i\bar S\wedge S.$$
 \end{Not}

 \subsection{Green functions}

To our knowledge the notion of Green function for a non-k\"ahlerian surface appears first in the paper \cite{DO}. It was further used in \cite{BrunellaGreen1}  and  in \cite{BrunellaGreen2}.

\begin{Def}\label{def:Green function} 
We say that a   
compact complex surface $X$ admits a  {\em Green function} if there exist a $\Z$-covering $\pi:X'\to X$, a divisor $D\ge0$ on $X$ and a negative plurisubharmonic function $G:X'\to]-\infty,0[$ which is multiplicatively automorphic on $X'$ and pluriharmonic on $X'\setminus\pi^{-1}(D)$. Being {\em multiplicatively automorphic} for $G$ means that if $g\in\Aut(X')$ generates the deck transformation group of 
$\pi:X'\to X$, there exists a positive constant $k$ such that $G\circ g=kG$. We will always implicitely assume that Green functions are non-trivial in the sense that $X'$ is connected and that $k\neq1$. By interchanging $g$ and $g^{-1}$ we may further assume that $k<1$.
\end{Def}

\begin{Prop}\label{prop:Green} 
If $(\pi,D,G)$ is data defining a Green function on a compact complex surface $X$ and if $u:=-\log(-G)$, then the following assertions hold:
\begin{enumerate}
\item $u$ is plurisubharmonic and {\em additively automorphic}. % with $u\circ g=u-\log(k)$. 
The additive automorphy for $u$ means that $u\circ g=u+p$, where $p:=-\log k$.
\item $i\partial\bar\partial u$ defines a non-trivial exact positive current on $X$ and in particular $X$ is non-k\"ahlerian.
\item $X$ is non-elliptic.
\item $i\partial\bar\partial G=\sum_j a_j[D_j]$, where $D_j$ are the irreducible components of  $\pi^{-1}(D)$ and $a_j$ are non-negative constants.
\item $u$ is in $L^2_{1,loc}(X')$ and
 $$i\partial\bar\partial u=i\partial u\wedge\bar\partial u.$$
 \item $I(i\partial\bar\partial u)=0$.
 \item For any continuous $p$-periodic function $h:\R\to\R$ satisfying $1+h'+h''\ge0$ as distributions, the function 
 $v:=u+h\circ u$, understood as being $-\infty$ on the polar locus of $u$, is plurisubharmonic, additive automorphic and defines an exact positive $(1,1)$-current $T:=i\partial\bar\partial v\in L^2_{-1}(X)$ with $I(T)=0$.
\end{enumerate}
\end{Prop}
\begin{proof}
 The function $\psi:]-\infty,0[\to\R$, $t\mapsto -\log(-t)$ is convex and increasing hence $u$ is plurisubharmonic. The assertions on the additive automorphic behaviour and on the fact that $i\partial\bar\partial u$ descends to a non-trivial exact positive current on $X$ are clear.
 
Suppose now by contradiction that $X$ is elliptic with elliptic fibration $f:X\to B$. By Liouville's theorem it follows that $G$ is constant on the connected components of the general fibers of $f\circ \pi:X'\to B$. Thus by the automorphic behaviour of $G$ the connected components of these general fibers are elliptic curves and $\pi$ factorizes through a $\Z$ covering $\pi':B'\to B$ of the base and a proper elliptic fibration $f':X'\to B'$. Clearly $G$ and $u$ descend then to plurisubharmonic functions on $B'$ with the corresponding automorphic behaviour. But as above this contradicts the fact that $B$ is K\"ahler.

Thus  $X$ is non-k\"ahlerian of algebraic dimension zero and the considerations in \cite[pp. 252-253]{BrunellaGreen2} apply to show that  $\pi^{-1}(D)$ is a divisor with simple normal crossings  and that $i\partial\bar\partial G=\sum_j a_j[D_j]$.

We now look at $u:=\psi\circ G$. By \cite{BlockiRemark}, \cite{Blocki} $u$ is in $L^2_{1,loc}$. Thus $\partial u=-\frac{\partial G}{G}$ is in $L^2_{loc}$,  
$$i\partial\bar\partial u=\frac{i\partial G\wedge\bar\partial G}{G^2}-\frac{i\partial\bar\partial G}{G},$$
and the last term vanishes since $i\partial\bar\partial G$ is an integration current over the polar locus of $G$, where $\frac{1}{G}$ vanishes, of course. We thus get 
$$i\partial\bar\partial u=i\partial u\wedge\bar\partial u,$$
$i\partial u\wedge\bar\partial u$ is $\d$-exact and hence $I(i\partial\bar\partial u)=0$.

Let finally $h$ be a $p$-periodic function satisfying $1+h'+h''\ge0$ as distributions and let
 $v:=u+h\circ u$. Away from the poles of $u$ we have $i\partial\bar\partial v=((1+h'+h'')\circ u )i\partial\bar\partial u$ and subharmonicity of $v$ here is a consequence of our assumption on $h$. By the mean value inequality $v$ is plurisubharmonic around the poles of $u$ as well.  
Since $h$ is continuous and periodic it will be bounded by some constant $C$ and we get $v\ge u-C$. Thus the singularities of $v$ are no worse than those of $u$, by \cite[Theorem 3.3]{Blocki}.  It remains to check that $I(T)=0$. 
For this note first that the condition  $1+h'+h''\ge0$ is equivalent to $(e^t+e^th')'\ge0$ and thus we may define an increasing function $f:[0,\infty[\to \R$ by $f(x):=(e^t+e^th')'([0,x])$, since $(e^t+e^th')'$ is a positive measure. It follows that the distribution $h'$ is represented by an $L^\infty$ function. 
Thus we can write $i\partial v\wedge\bar\partial v=(1+h'\circ u)^2 i\partial u\wedge\bar\partial u$. We shall exhibit a positive constant $ \mu$ and a continuous $p$-periodic function $g$ on $\R$ such that
 \begin{equation}\label{eq:H}
 (1+h'\circ u)^2 i\partial u\wedge\bar\partial u=i\partial\bar\partial(\mu u+g\circ u).
 \end{equation}
  Put 
 $H:=(1+h'\circ u)^2$, $\mu:=\frac{1}{p}\int_0^pH(s)\d s$, $C:=\frac{1}{e^p-1}\int_0^p(e^sH(s)-e^s\mu)\d s$ and $g(t):=\int_0^t(H(s)-\mu)\d s -e^{-t}\int_0^t(e^sH(s)-e^s\mu)\d s +C(1-e^{-t})$. Then $\mu$ and $g$ fulfill the desired conditions and thus $i\partial v\wedge\bar\partial v$ is $\d$-closed and hence $I(T)=0$.
 \end{proof}

In fact it will follow from the work of Brunella in \cite{BrunellaGreen1}, \cite{BrunellaGreen2} and from our  Proposition \ref{prop:hyperbolic} that if $X$ admits a Green function then all exact positive $(1,1)$-currents on $X$ are up to a multiplicative factor of the form $i\partial\bar\partial v$ for an additively automorphic function $v$ as above, see Corollary \ref{cor:structureGreen}. 
%%%%%%%%%%%%%%%%%%%%%%%%%%%%%%%%%%%%%%%%%%%%%%%%%%%%%%%%%
\section{Classification of non-k\"ahlerian surfaces from a dynamical point of view}\label{section:classification}
\subsection{The known classes of non-k\"ahlerian surfaces}\label{subsection:historical_classification}
The known minimal non-k\"ahlerian surfaces may be divided into the following classes:
\begin{enumerate}
\item minimal elliptic non-k\"ahlerian surfaces,
\item non-elliptic Hopf surfaces,
\item Inoue surfaces,
\item Kato surfaces.
\end{enumerate}
Note that any non-k\"ahlerian surface admits a unique minimal model \cite[Theorem VI.1.1]{BHPV}. Here we will give a short description of each class; see \cite{NakamuraSugaku} for a detailed exposition.

\subsubsection{Minimal elliptic non-k\"ahlerian surfaces}\label{subsubsection:elliptic}
These are by definition minimal surfaces $X$ with odd first Betti number, admitting a fibration $\pi:X\to Y$ with elliptic general fibers onto a curve $Y$. It can be shown \cite[Proposition 3.17]{Brinzanescu} that in this case the fibration $\pi$ is a {\em quasi-bundle}, i.e. all its smooth fibers are pairwise isomorphic and its singular fibers are multiples of smooth elliptic curves. From loc. cit. it also follows that $h^{1,0}(X)=h^{1,0}(Y)$, i.e. all holomorphic $1$-forms on $X$ are pull-backs of holomorphic $1$-forms on $Y$, see also the proof of the next proposition.

 \begin{Prop}\label{prop:elliptic} 
If $X$ is a minimal elliptic  non-k\"ahlerian surface, then the following assertions hold:
\begin{enumerate}
\item Every positive divisor $D$ on $X$ is a positive combination with rational coefficients of fibers of $\pi$ and is homologically trivial. In particular there exist exact positive $(1,1)$-currents on $X$ not  in $L^2_{-1}$.
\item All exact positive $(1,1)$-currents $T$ which are in   $L^2_{-1}$ necessarily have $I(T)>0$.
\end{enumerate}
\end{Prop}
\begin{proof}
 The first assertion is clear. 
 
 Let now $T=i\partial S$ be an exact positive $(1,1)$-current on $X$, with $S$ a $(0,1)$-current with coefficients in $L^2(X)$. Let $\omega_Y$ a volume form on $Y$. Then $\omega_X:=\pi^*\omega_Y$ is positive non-trivial and such that $\omega_X\wedge\omega_X=0$. Thus $\{\omega_X\}=c\{\tau\}\in H^{1,1}_{BC}(X,\R)$ for some positive real number $c$. Suppose that
 \begin{equation}\label{eq:I}
0=I(T):=\int i\bar S\wedge S\wedge \omega_X.
\end{equation}
 We shall show that $T=0$.
 
 Let $Y^\circ$ be the set of regular values of $\pi$ and set $X^\circ:=\pi^{-1}(Y^\circ)$. We will begin by working on $X^\circ$. Since $\pi:X\to Y$ is a quasi-bundle it follows that the fibration $\pi^\circ: X^\circ\to Y^\circ$ is locally trivial over $Y^\circ$. For such a local trivialization we choose local coordinates $(z,w)$ on $X^\circ$ where $z$ is a local coordinate on $Y^\circ$ and $w$ is a coordinate for the fiber direction. The formula \eqref{eq:I} implies that $S=f\d \bar z$ where $f$ is locally in $L^2$ on $X^\circ$. Since $T$ is real and $T=i\partial S$ we also get $\frac{\partial f}{\partial w}=0$ as distributions. Since $\bar\partial S=0$ we further get $\frac{\partial f}{\partial \bar w}=0$. Thus the distribution $f$ is independent of the $w$ coordinate and it follows that $f$ is a tensor product of the function $1$ in the vertical direction with an $L^2_{loc}$-function $f^\circ$ on $Y^\circ$, cf. \cite[IV.5.Exemple 1]{Schwartz}. Setting $R^\circ=f^\circ\d z$ on $Y^\circ$ we may say that $S$ "comes from $R^\circ$ from the base", meaning by this that $S$ is the tensor power of the function $1$ in fiber direction with $R^\circ$ in horizontal direction. The form $R^\circ$ has coefficients in $L^2_{loc}(Y^\circ)$. Moreover, $T$ "comes from $i\partial R^\circ$ from the base", in particular $i\partial R^\circ$ is a positive $(1,1)$-current on $Y^\circ$. We shall next show that it admits an extension to $Y$ as a positive exact $(1,1)$-current. From this it will follow that $i\partial R^\circ=0$.
 
 We look at the situation around a singular fiber of $\pi$ over some critical value $y_0\in Y$.
 By \cite[Proposition III.9.1 and p.207]{BHPV} we know that over a small neighbourhood $V$ of $y_0$ in $Y$ the restriction $X_V$ of $X$ may be seen as the quotient $p:\T\times \D\to X_V$ of $\T\times \D$ by the action of $\Z/n\Z$ generated by $(w,z)\mapsto (w+1/n,\rho z)$ where $\T$ is a one dimensional complex torus given as $\C/\Lambda$, $\Lambda$ is the lattice generated by $1$ and some $\alpha\in\mathbb{H}$, and $\rho= \exp(\frac{2 i\pi}{n})$. Supposing that $V$ is byholomorphic to $\D$ we thus get a commutative diagram
  $$
 \xymatrix{\T\times \D \ar[r]^{p} \ar[d]^{pr_2} & X_V \ar[d]^{\pi} \\
\D\ar[r]^{\phi} & V, }
$$ 
 where $\phi(z)=z^n$. Note that $p$ is an unramified covering map. We set $\tilde S:=p^*S$, $\tilde T:=p^*T$ and we have as before $\int i\bar{\tilde S}\wedge\tilde S\wedge pr_2^*\phi^*\omega_V=0$ hence  $\tilde S=\tilde f \d z$ for some $L^2$ function $\tilde f$ on $\T\times \D$ not depending on the vertical variable. Thus there exists a $(0,1)$-current $\tilde R\in L^2_{loc}(\D)$ such that $\tilde S$ and $\tilde T$ "come from $\tilde R$ and $i\partial\tilde R$ respectively from the base". 
 In particular $i\partial\tilde R$ is a positive $(1,1)$-current on $\D$. Set $R:=\phi_*\tilde R$. We next show that on $V^*:=V\setminus\{ y_0\}$ we have
 $$R_{|V^*}=R^\circ.$$
 We set $\zeta=z^n=\phi(z)$ in $V$. Note that if we write $R^\circ=f^\circ\d\bar\zeta$ on $V^*$, we have $\phi^*R^\circ=n\bar z^{n-1}(f^\circ\circ\phi)\d z=\tilde f\d z$ on $\D^*$. Then for a $(1,0)$-form $\eta=g\d\zeta$ on $V^*$ we get
 $<\phi_* \tilde R,\eta>=\int_\D nz^{n-1}\tilde f \cdot(g\circ\phi)\d\bar z \wedge \d z =
   \int_\D n^2|z|^{2(n-1)}(f^\circ\circ\phi) \cdot(g\circ\phi)\d\bar z \wedge \d z =  \int_V f^\circ\cdot g\d \bar\zeta\wedge\d\zeta=<R^\circ,\eta> $, hence $R_{|V^*}=R^\circ.$
 Thus $i\partial R=\phi_*(i\partial\tilde R)$ is a positive exact current on $V$ extending the current $i\partial R^\circ$. Since it can be considered on the whole $Y$, this extension must be trivial. Thus $T$ itself is trivial on $X^\circ$. But then $T$ is concentrated on a finite number of fibers of $\pi$. Unless $T=0$ this contradicts the assumption $T\in L^2_{-1}(X)$ and the proof is finished.  
 \end{proof}

\subsubsection{Non-elliptic Hopf surfaces}\label{subsubsection:Hopf} A compact complex surface $X$ is said to be a {\em Hopf surface} if its universal covering space is isomorphic to %$\W:=
$\C^2\setminus \{0\}$. A Hopf surface is called {\em primary} if its fundamental group is infinite cyclic, and {\em secondary} otherwise. The following facts on Hopf surfaces $X$ and much more were shown by Kodaira in \cite{KodairaStructureII}:
\begin{enumerate}
\item If $X$ is a primary Hopf surface, then its fundamental group is generated by a {\em contraction} $g:\C^2\setminus \{0\}\to\C^2\setminus \{0\}$ which for suitable global holomorphic coordinates $(z_1,z_2)$ on $\C^2$ has the following {\em normal form}
\begin{equation}\label{eq:contraction}
g(z_1,z_2)=(\alpha_1z_1+\lambda z_2^m, \alpha_2 z_2),
\end{equation}
where $m\in\Z_{>0}$, $\alpha_1,\alpha_2,\lambda\in\C$ and 
$$(\alpha_1-\alpha_2^m)\lambda=0, \ 0<|\alpha_1|\le|\alpha_2|<1.$$
\item A primary Hopf surface $X=(\C^2\setminus \{0\})/<g>$ with $f$ as above is elliptic if and only if $\lambda=0$ and $\alpha_1^{k_1}=\alpha_2^{k_2}$ for some positive integers $k_1$, $k_2$.
\item If $X=(\C^2\setminus \{0\})/\pi_1(X)$ is a non-elliptic secondary Hopf surface then its fundamental group $\pi_1(X)$ is isomorphic to $\Z\times(\Z/l\Z)$ where the direct factor $\Z$ is generated by a contraction $g$ of the form  \eqref{eq:contraction} and the finite cyclic group $\Z/l\Z$ is generated by an automorphism of $\C^2\setminus \{0\}$ of the form
$$(z_1,z_2)\mapsto(\epsilon_1 z_1, \epsilon_2 z_2),$$
where $\epsilon_1$, $\epsilon_2$ are primitive $l$-th roots of unity satisfying.
$$(\epsilon_1-\epsilon_2^m)\lambda=0.$$
In particular $X$ admits a finite unramified cyclic covering by the primary Hopf surface $(\C^2\setminus \{0\})/<g>$. 
\item $b_1(X)=1$ and $b_2(X)=0$.
\item Non-elliptic Hopf surfaces contain one or at most two irreducible compact curves according to whether $\lambda\neq0$ or $\lambda=0$, for $\lambda$ as in equation \eqref{eq:contraction}. These curves are elliptic.
\end{enumerate}

In their study of closed positive $(1,1)$-currents on compact complex surfaces done in \cite{HarveyLawson}, Harvey and Lawson subdivide non-elliptic primary Hopf surfaces into two classes. Their definitions are immediately extended to secondary non-elliptic Hopf surfaces too as follows:
\begin{enumerate}
\item {\em Class $1$} contains those non-elliptic Hopf surfaces for which the coefficient $\lambda$ in the above formulas vanishes. (Thus this class contains exactly those Hopf surfaces admitting precisely two elliptic curves.)
\item {\em Class $0$} contains those non-elliptic Hopf surfaces for which  $\lambda\neq0$. (These are the Hopf surfaces containing only one elliptic curve.)
\end{enumerate}

\begin{Prop}\label{prop:Hopf} 
\begin{enumerate}
\item
Up to a non-negative factor there exists exactly one closed positive $(1,1)$-current on a non-elliptic Hopf surface of class $0$. This is the integration current along the elliptic curve of the surface. 
\item
Every non-elliptic Hopf surface $X$ of class $1$ admits non-trivial closed positive $(1,1)$-currents  $T$ in $L^2_{-1}(X)$ and for such currents one always has $I(T)>0$.
\end{enumerate}
\end{Prop}
\begin{proof}
The assertion on Hopf surface of class $0$ was proved in \cite[Theorem 69]{HarveyLawson} for primary Hopf surfaces. The case of the secondary Hopf surfaces immediately follows from this by  pull-back and push-forward through the finite covering map $(\C^2\setminus \{0\})/<g>\to
(\C^2\setminus \{0\})/(\Z\times(\Z/l\Z))$.

In the same way it will be enough to establish the second assertion only for primary non-elliptic Hopf surfaces  of class $1$. Let $X$ be such a surface given by a contraction $g$ of the form
$$g(z_1,z_2)=(\alpha_1z_1, \alpha_2 z_2),$$
with $0<|\alpha_1|\le|\alpha_2|<1$. The existence of non-trivial closed positive $(1,1)$-currents  in $L^2_{-1}(X)$ follows from \cite[Theorem 58]{HarveyLawson}, where it is even proved that smooth such currents exist. More precisely in \cite{HarveyLawson} Harvey and Lawson consider the following objects on  $\C^2\setminus \{0\}$  some of which obviously descend to $X$. Set 
$$r=\frac{\log|\alpha_1|}{\log|\alpha_2|},$$  
$$\phi: \C^2\setminus \{0\}\to \R, \ \phi(z_1,z_2):=\log(|z_1|^2+|z_2|^{2r}),$$
$$\eta:=z_2\d z_1-rz_1\d z_2.$$
$$\Omega:=i\partial\bar\partial\phi=
\frac{|z_2|^{2(r-1)}}{(|z_1|^2+|z_2|^{2r})^2}i\eta\wedge\bar\eta,$$
$$V:=rz_1\frac{\partial}{\partial z_1}-
z_2\frac{\partial}{\partial z_2},$$
$$\pi:X \to [0,1], \ 
\pi(z_1,z_2):=\frac{|z_1|^2}{|z_1|^2+|z_2|^{2r}}.$$ 
It is said in \cite{HarveyLawson} that the form $\Omega$ is smooth on $X$ but this might not be the case around the elliptic curve $E_1:=\{z_2=0\}$ when $r\notin\N$. To remedy to this one may consider 
$$r'=\frac{1}{r},$$  
$$\phi': \C^2\setminus \{0\}\to \R, \ \phi'(z_1,z_2):=\log(|z_2|^2+|z_1|^{2r'}),$$
$$\eta':=z_1\d z_2-r'z_2\d z_1=-r'\eta.$$
$$\Omega':=i\partial\bar\partial\phi'=
\frac{|z_1|^{2(r'-1)}}{(|z_2|^2+|z_1|^{2r'})^2}i\eta'\wedge\bar\eta',$$
$$\pi':X\to [0,1], \ 
\pi'(z_1,z_2):=\frac{|z_1|^2}{|z_1|^2+|z_2|^{2r}},$$
and 
$$\tilde \Omega:=(\psi\circ\pi)\Omega+(\psi\circ\pi')\Omega',$$
where $\psi:[0,1]\to[0,1]$ is smooth and equals $1$ in a neighbourhood of $0$ and $1$ in a neighbourhood of $1$. Then $\tilde \Omega$ is a smooth positive $\d$-closed $(1,1)$-form on $X$ without zeroes on $X$. %vanishing only along the elliptic curves $E_1:=\{z_2=0\}$, $E_2:=\{z_1=0\}$. 
The $\d$-closedness of $\tilde\Omega$ follows from the fact that $\partial\pi$ and $\partial\pi'$ are proportional to $\eta$ on $X\setminus E_1$ and on $X\setminus E_2$, respectively. 

Note further that the holomorphic vector field $V$ defines a holomorphic foliation $\cF$ on $X$, which coincides with the complex foliation defined by $\tilde \Omega$ and whose leaves are dense in the fibers of $\pi$ as is shown in \cite[Lemma 54]{HarveyLawson}.  

Let now $T$ be a non-trivial closed positive $(1,1)$-current in $L^2_{-1}(X)$.
By \cite[Proposition 4]{TomaCurrents} there exists an additively automorphic $i\partial\bar\partial$-potential $u$ of $T$ in $L^2_{1, loc}(\C^2\setminus \{0\})$.  Supposing by contradiction that $I(T)=0$, we infer that $i\partial u\wedge\bar\partial u\wedge\tilde{\Omega}=0$ on $X$ and also that $\partial u\wedge\tilde{\Omega}=0$ and $\bar\partial u\wedge\tilde{\Omega}=0$. Thus the restriction of $u$ to those leaves of $\cF$ not contained in the polar set of $u$ is subharmonic and in fact constant. 
By semi-continuity of $u$ and since the closures of the leaves of $\cF$ are fibers of $\pi$ it follows that $u$ is constant on these fibers as well. Thus $u$ has trivial additive automorphy and $T=0$. This is a contradiction.
\end{proof}

\subsubsection{Inoue surfaces}\label{subsubsection:Inoue}  In this paper by an {\em Inoue surface} we understand a  compact complex surface $X$ with $b_1(X)=1$, $b_2(X)=0$  and no compact complex curves. The construction of Inoue surfaces appears in \cite{Inoue} and their classification was completed in \cite{TelemanInoue} %primit pe 7/7/93
 and in \cite{LiYauZheng}. %primit pe 12/7/93
Their universal cover is $\bH\times\C$ and their universal group is generated by four affine transformations $g_0$, $g_1$, $g_2$, $g_3$ in such a way that $\pi_1(X)$ appears as a semidirect product $\Gamma\rtimes<g_0>$ of $\Gamma$ by $<g_0>$, where $\Gamma$ is the subgroup generated by $g_1$, $g_2$, $g_3$, and $g_0$ acts on $\bH\times\C$ by
$$g_0(w,z)=(\alpha w, \beta z+t),$$
for some positive real number $\alpha<1$ and suitable complex numbers $\beta$ and $t$. Moreover for $i=1,2,3$ the elements $g_i$ act on $\bH\times\C$ by
$$g_i(w,z)=( w+a_i, z+b_i w+c_i),$$
for some real numbers $a_i$ $b_i$ and complex numbers $c_i$, see \cite{Inoue}. Here $w$ and $z$ denote complex coordinates on $\bH$ and on $\C$ respectively. Thus the quotient group $\pi_1(X)/\Gamma$ is infinite cyclic generated by the class $\hat g_0$ of $g_0$, defines a $\Z$-covering $\pi:X'\to X$ of $X$ and the function $y:=\Im m (w)$ defined on $\bH\times\C$ descends to a function $\hat y:X'\to\R$. 

\begin{Prop}\label{prop:Inoue} If $X$ is an Inoue surface, then under the  above notations putting $G:=-\hat y$ we get a Green function 
$G:X'\to\R$ without poles on $X$.  Moreover if $u:=-\log(-G)$ and  
$p:=-\log \alpha$, then up to a multiplicative factor any non-trivial closed positive $(1,1)$-current $T$ on $X$ is of the form $T=i\partial\bar\partial v$, where $ v:=u+h\circ u$ for some continuous $p$-periodic function $h:\R\to\R$ satisfying $1+h'+h''\ge0$ as distributions. All such currents are in 
 $L^2_{-1}(X)$ and have $I(T)=0$.
\end{Prop}
\begin{proof}
The fact that $G$ is a Green function without poles is clear. By \cite[Theorem 82]{HarveyLawson} every closed positive $(1,1)$-current $T$ on $X$ is of the form $T=(\phi\circ u)i\partial\bar\partial u$, where $\phi$ is a positive $p$-periodic generalized function on $\R$. We may see $\phi$ as a  $p$-periodic (positive) measure on $\R$ and we may assume that $\phi(]0,p])=p$. In order to find the desired function $h$ it suffices to solve the equation
$$1+h'+h''=\phi$$
on $\R$. For this  remark first that $1+h'+h''=\phi$ is equivalent to $(e^t+e^th')'=e^{t}\phi$. Integrating once gives us a right-continuous increasing function $f:\R\to \R$  of bounded variation such that  $f(x):=(e^{t}\phi)(]0,x])$  for all $x\in \R$, \cite{FunctionBoundedVariation}. A second integration leads to the desired continuous $p$-periodic function $h$. 

The assertion on the regularity of $T$ and on $I(T)$ follows now from Proposition \ref{prop:Green}. 
\end{proof}

\subsubsection{Kato surfaces}\label{subsubsection:Kato}
A {\em Kato surface } is a minimal surface $X$ with $b_1(X)=1$, $b_2(X)>0$ and admitting a {\em global spherical shell}, that is an open neighbourhood $\Sigma$ of the $3$-dimensional sphere $S^3$ in $\C^2\setminus\{0\}$ holomorphically embedded in $X$ and such that $X\setminus\Sigma$ is connected. Their construction is due to Masahide Kato, \cite{Kato}, and their properties have been studied by many authors. 

Any Kato surface $X$ admits exactly $b_2(X)$ rational curves. Conversely, if a minimal non-k\"ahlerian surface $X$ admits $b_2(X)$ rational curves, then $X$ is a Kato surface.

The class of Kato surfaces contains subclasses of previously constructed surfaces known as parabolic Inoue surfaces \cite{InoueParabolic} and Inoue-Hirzebruch surfaces, also called hyperbolic Inoue surfaces \cite{InoueHyperbolic}. We will not use the terminology "`parabolic Inoue"' and "`hyperbolic Inoue"'in order not to create confusion with the already described class of Inoue surfaces. The reader may consult \cite{NakamuraSugaku} for an account of these surfaces. We prefer instead to consider the following subclassification of Kato surfaces: 
\begin{enumerate}
\item {\em Enoki surfaces}, which are non-k\"ahlerian compactifications of affine line bunles over elliptic curves by cycles $D$ of rational curves. Enoki shows that these surfaces are Kato surfaces, that $(D^2)=0$, and that, conversely, any minimal; surface with $b_1=1$, $b_2>0$ and with a non-trivial divisor $D$ with $(D^2)=0$ is in this subclass, \cite{Enoki}.
\item {\em Inoue-Hirzebruch surfaces}, which are Kato surfaces whose rational curves are organized in one or two cycles.
\item {\em Intermediate Kato surfaces}, which are Kato surfaces whose divisor of rational curves is a cycle with at least one branch attached.
\end{enumerate}

\begin{Prop}\label{prop:Kato} 
\begin{enumerate}
\item
On an Enoki surface there exists exactly one exact positive $(1,1)$-current  up to a positive multiplicative factor. This is the integration current along the reduced  divisor of rational curves  of the surface. 
\item 
If  $X$ is an Inoue-Hirzebruch surface or an intermediate Kato surface, then  $X$ admits a Green function $G$. Moreover if $u:=-\log(-G)$ is the associated additively automorphic plurisubharmonic function with $u\circ g=u+p$ and $<g>=\pi_1(X)$, then up to a multiplicative factor any non-trivial exact positive $(1,1)$-current $T$ on $X$ is of the form $T=i\partial\bar\partial v$, where $ v:=u+h\circ u$ for some continuous $p$-periodic function $h:\R\to\R$ satisfying $1+h'+h''\ge0$ as distributions. All such currents are in 
 $L^2_{-1}(X)$ and have $I(T)=0$.
\end{enumerate}
\end{Prop}
\begin{proof}
The first statement is part of \cite[Theorem 10]{TomaCurrents}. The existence of Green functions on intermediate Kato and on Inoue-Hirzebruch surfaces was shown in \cite{DO}. A complete description of the exact positive $(1,1)$-currents on these surfaces was given in \cite[Theorem 11, Theorem 12]{TomaCurrents}.
\end{proof}

\subsection{Hyperbolic and parabolic non-k\"ahlerian surfaces}\label{subsection:new_classification}
The next definition divides the known classes of non-k\"ahlerian surfaces into two groups: {\it parabolic} surfaces and {\it hyperbolic} ones. We will then show  that members of each of these groups have many properties in common. One may speculate which of these properties are better suited to approach the Global Spherical Shell Conjecture.  

\begin{Def}\label{def:hyperbolic-parabolic} 
A non-k\"ahlerian compact complex surface $X$ will be said to be  {\em parabolic} if it belongs to one of the classes: Hopf surfaces, Enoki surfaces, non-k\"ahlerian elliptic surfaces. 
It will be said to be  {\em hyperbolic} if $X$ is either an Inoue surface, an Inoue-Hirzebruch surface, or an intermediate Kato surface.
\end{Def}

This terminology is first used by Inoue in the particular cases of the examples of non-k\"ahlerian surfaces that he constructs in \cite{InoueParabolic} and in \cite{InoueHyperbolic}. Note that non-k\"ahlerian surfaces that are hyperbolic in the above sense are not hyperbolic according to the standard terminology used in complex geometry, as they have many entire curves.

\begin{Prop}\label{prop:hyperbolic} 
If $X$ is a hyperbolic  non-k\"ahlerian surface, then the following assertions hold:
\begin{enumerate}
\item $X$ admits a Green function $G$ such that the function $u=\psi(G):=-\log(-G)$ is in $L^2_{1,loc}$ and
 $$i\partial\bar\partial u=i\partial u\wedge\bar\partial u.$$
\item All exact positive $(1,1)$-currents $T$ are of the form $T=\lambda i\partial\bar\partial ( u+ h\circ u)$ with $\lambda\ge0$, $h:\R\to\R$ a continuous $p$-periodic function  satisfying $1+h'+h''\ge0$ and $p$ the automorphy summand of $u$ as in Proposition \ref{prop:Green}.
Moreover all these currents are in $L^2_{-1}$ and have $I(T)=0$. 
\item The only homologically trivial divisor on $X$ is $0$.
\item $\widetilde{X\setminus D}_{max}\cong\D\times\C$.
\end{enumerate}
\end{Prop}

\begin{proof}
The assertions on the Green functions and on the exact positive currents follow from Propositions \ref{prop:Green}, \ref{prop:Inoue} and \ref{prop:Kato}.

The assertion on the homologically trivial divisors follows from the knowledge of the structure of the reduced divisor of curves on these surfaces, cf. \cite{NakamuraSugaku}.  

Finally the facts on the universal cover of $X\setminus D_{max}$ are established in \cite{Inoue}, \cite{InoueHyperbolic} and \cite[Theorem 3.7]{DOT3}. 
\end{proof}

\begin{Prop}\label{prop:parabolic} 
If $X$ is a parabolic  non-k\"ahlerian surface, then the following assertions hold:
\begin{enumerate}
\item $X$ admits no Green function.
\item All exact positive $(1,1)$-currents $T$ in   $L^2_{-1}(X)$ necessarily have $I(T)>0$.
\item There exist homologically trivial divisors $D$ on $X$ with $D>0$, and in particular there exist exact positive $(1,1)$-currents on $X$ not  in $L^2_{-1}$.
\item If the algebraic dimension of $X$ is zero, then $\widetilde{X\setminus D}_{max}\cong\C^2$ and in particular there exists no divisor $D$ on $X$ such that 
$\widetilde{X\setminus D}\cong\D\times\C$ in this case.
\end{enumerate}
\end{Prop}
\begin{proof}
If a non-k\"ahlerian surface $X$ admits a Green function then $X$ is non-elliptic by Proposition \ref{prop:Green} and thus of algebraic dimension zero. In this case it is shown by Brunella in 
\cite{BrunellaGreen1} and \cite{BrunellaGreen2} that $X$ is necessarily hyperbolic. In particular, parabolic surfaces will not admit Green functions.

The assertions on  the exact positive currents  and on the homologically trivial divisors follow from Propositions  \ref{prop:elliptic}, \ref{prop:Hopf} and \ref{prop:Kato}.

Finally, for the two classes of parabolic surfaces of algebraic dimension zero, namely for non-elliptic Hopf surfaces and for Enoki surfaces, it follows almost from the definition that the universal cover of the complement of the union of compact complex curves is isomorphic to $\C^{2}$. 
\end{proof}

\begin{Th}\label{thm:main}
Non-k\"ahlerian compact complex surfaces $X$ may be classified according to the following table. 
\begin{tabular}{|l|l|l|} %|l|c|r|}
  \hline
Criterion $C$ & $X$ satisfying $C$  & $X$ not satisfying $C$\\
  \hline
$(D^2)<0$ \ \ $\forall D\in\Div(X)\setminus\{0\}$&hyperbolic, ? & parabolic \\
all exact $(1,1)$-currents on $X$ are in $L^2_{-1}$ & hyperbolic, ? &parabolic,  ?\\
all exact currents  $T\in L^2_{-1}(X)$ have $I(T)=0$ & hyperbolic, ? &parabolic,  ?\\
%all exact currents  $T\in L^2_{-1}(X)$ have $I(T)>0$ & parabolic, ? &hyperbolic,  ?\\
$X$ admits a Green function& hyperbolic & parabolic,  ?\\
$a(X)=0$ and $\exists D$ with $\widetilde{X\setminus D}\cong\D\times\C$& hyperbolic & parabolic,  ?\\
  \hline
\end{tabular}
The question marks signal that possibly not yet known surfaces may respond to the corresponding criteria.
\end{Th}
\begin{proof}
The presence of hyperbolic and parabolic surfaces at the indicated places of the table is a consequence of the Propositions \ref{prop:hyperbolic} and \ref{prop:parabolic}. We are left  only with the task of explaining the absence of question marks at three places of the table.

The fact that parabolic surfaces are the only compact complex surfaces admitting homologically trivial divisors $D$ with $D>0$  is due to Enoki, \cite{Enoki}. 

Non-k\"ahlerian non-elliptic surfaces admitting Green functions have been shown to be hyperbolic by Brunella  in 
\cite{BrunellaGreen1} and \cite{BrunellaGreen2}. The case of elliptic surfaces is settled by Proposition \ref{prop:Green}.

Finally, it is again Brunella who proved in \cite{BrunellaCylinder} that hyperbolic surfaces are the only non-k\"ahlerian non-elliptic surfaces whose complement of the maximal divisor of curves is uniformized by $\D\times\C$.
\end{proof}

Combining the Theorem and Proposition \ref{prop:hyperbolic} one immediately gets the following
\begin{Cor}\label{cor:structureGreen}
If $X$ admits a Green function $G$ and  if $u:=-\log(-G)$ is the associated additively automorphic plurisubharmonic function with $u\circ g=u+p$ and $<g>=\pi_1(X)$, then up to a multiplicative factor any non-trivial exact positive $(1,1)$-current $T$ on $X$ is of the form $T=i\partial\bar\partial v$, where $ v:=u+h\circ u$ for some continuous $p$-periodic function $h:\R\to\R$ satisfying $1+h'+h''\ge0$ as distributions. 
\end{Cor}

%%%%%%%%%%%%%%%%%%%%%%%%%%%%%%%%%%%%%%%%%%%%%%%%

%%%%%%%%%%%%%%%%%%%%%%%%%%%%%%
%%%%%%%%%%%%%%%%%%%%%%%%%%%%%%%%%%%%%%%%%%%%%%%%%%%%%%%%%
\section{Perspectives}\label{sec:perspectives}

In this section we wish to briefly discuss a number of conjectures and questions related to the degree of regularity of the $\d$-closed positive $(1,1)$-currents on a compact non-k\"ahlerian surfaces. The leading idea is the same which guided our approach to the study of the K\"aler rank of surfaces in \cite{ChiTo}. In that paper we worked under the assumption that a non-trivial positive smooth exact $(1,1)$-current $T$ exists on a compact complex surface $X$ and we aimed at a classification by distinguishing two cases according to whether $I(T)$ is positive or zero. In the first case we showed that $X$ was necessarily elliptic or Hopf of class $1$. In the second case we proved that $X$ admitted a Green function without poles. This case was afterwards completely settled by Brunella in \cite{BrunellaGreen1}, who showed that such Green functions were only supported by Inoue surfaces. Trying to extend this type of strategy and in view of the striking similarities exhibited by Theorem  \ref{thm:main} for the surfaces which are hyperbolic or respectively parabolic we are led to the following conjectures.

\begin{Conj}
If $X$ is a non-k\"ahlerian surface all of whose exact positive $(1,1)$-currents $T$ are in $L^2_{-1}(X)$ and satisfy $I(T)=0$, then $X$ admits a Green function, and in particular $X$ is hyperbolic.
\end{Conj}

\begin{Conj}
If $X$ is a non-k\"ahlerian surface all of whose exact positive $(1,1)$-currents $T$ are in $L^2_{-1}(X)$ but do not all satisfy $I(T)=0$,   then $X$ admits a cycle of rational curves.
\end{Conj}

\begin{Conj}
If $X$ is a non-k\"ahlerian surface admitting an exact positive $(1,1)$-current not in $L^2_{-1}(X)$, then there exists on $X$ some exact positive current $T$ with a non-vanishing Lelong number at at least one point of $X$.
\end{Conj}
Note that in this case the surface $X$ would be parabolic. Indeed, for such a current the Lelong-Siu level sets $E_{c}(T)$ cannot all be zero-dimensional, by  \cite[Theorem A.1]{TelemanFamilies}. Thus supposing that $X$ is non-elliptic, we would get $T=[C]+R$ with $C$ a curve and $R$ is residual, and $R$ would be $\d$-closed and nef by Proposition \ref{prop:decomposition} and thus $\d$-exact. Therefore $[C]$ would also be $d$-exact, which implies the parabolicity of $X$.

%%%%%%%%%%%%%%%%%%%%%%%%%%%%%%%%%%%%%%%%%%

\section{Appendix}\label{sec:appendix}
Since the notion of nef current is not used frequently in the literature we present here some properties relating it to the more common notion of nef class.

As before, also in this section we denote by $X$ a compact non-k\"ahlerian surface. 

Let $\cE^{p,q}$ and $\cD'^{p,q}$ be the sheaves of germs of smooth $(p,q)$-forms and respectively
of bidegree $(p,q$-currents on $X$. We will write $\cE^{p,q}_{\R}$ and $\cD'^{p,q}_{\R}$ for the subsheaves of real forms, and respectively real currents.
We will be interested in the real Bott-Chern and Aeppli cohomolgy groups of bidegree $(1,1)$ on $X$. They may be defined using either global forms or global currents. We recall their definition in terms of forms:
$$H^{1,1}_{BC}(X,\R):=\{\eta\in\cE_{\R}^{1,1}(X) \ | \ \d\eta=0\}/i\partial\bar\partial\cE_{\R}^{0,0}(X),$$
$$H^{1,1}_{A}(X,\R):=\{\eta\in\cE_{\R}^{1,1}(X) \ | \ i\partial\bar\partial\eta=0\}/\{\bar\partial S + \partial\bar S \ | \  S\in \cE^{1,0}(X) \}.$$
The evaluation of currents on forms gives a duality between these two spaces. 
We also get natural comparison morphisms to and from the second de Rham cohomology group:
$$ H^{1,1}_{BC}(X,\R)\to H^{2}_{dR}(X,\R), \ H^{2}_{dR}(X,\R)\to H^{1,1}_{A}(X,\R).$$ We denote the image of the first one by $H^{1,1}_{dR}(X,\R)$. We clearly have
$$H^{1,1}_{dR}(X,\R)=\{\eta\in\cE_{\R}^{1,1}(X) \ | \ \d\eta=0\}/\{\eta\in\cE_{\R}^{1,1}(X) \ | \ \eta=\d\phi, \ \phi\in \cE_{\R}^{1}(X) \}.$$
It is known that  
$\Ker (H^{1,1}_{BC}(X,\R)\to H^{1,1}_{dR}(X,\R))$ and $\Coker (H^{1,1}_{dR}(X,\R)\to H^{1,1}_{A}(X,\R))$ are $1$-dimensional, \cite{BHPV}. They are generated by $\{\tau\}_{BC}$ and by the image of $\{\omega\}_{A}$ respectively. It is also known that the intersection form on $H^{1,1}_{dR}(X,\R)$ is negative definite, \cite{BHPV}. 

We next define ``positive'' convex cones in $H^{1,1}_{BC}(X,\R)$ and in $H^{1,1}_{A}(X,\R)$ by
$$\Psef_{BC}(X):=\{ \{ T\}_{BC} \ | \ T\in {\cD'}_{\R}^{1,1}(X), \ \d T=0, \ T\ge0\},$$
$$\Psef_{A}(X):=\{ \{ T\}_{A} \ | \ T\in {\cD'}_{\R}^{1,1}(X), \ i\partial\bar\partial T=0, \ T\ge0\},$$
$$\Nef_{BC}(X):=\{ \{ T\}_{BC} \ | \ T\in {\cD'}_{\R}^{1,1}(X), \ \d T=0, \ T \ \text{nef}\},$$
$$\Nef_{A}(X):=\{ \{ T\}_{A} \ | \ T\in {\cD'}_{\R}^{1,1}(X), \ i\partial\bar\partial T=0,  \ T \ \text{nef}\}.$$
We also denote by $G$ the set of Aeppli cohomology classes of Gauduchon forms on $X$.

\begin{Prop}\label{prop:psef-nef} 
\begin{enumerate}
\item $\Psef_{BC}(X)$ and $\Nef_{BC}(X)$ are closed in $H^{1,1}_{BC}(X,\R)$.
\item 
$\Nef_{BC}(X)=\{\alpha\in H^{1,1}_{BC}(X,\R) \ | \ \forall \epsilon>0 \ \exists \eta_{\epsilon}\in \alpha\cap\cE_{\R}^{1,1}(X) \ \eta_{\epsilon}\ge-\epsilon\omega\}$.
\item 
$\Nef_{BC}(X)=\R_{\ge0}\{\tau\}_{BC}$.
\item
 If the Bott-Chern cohomology class of a positive closed current $T$ is in $\Nef_{BC}(X)$, then $T$ is nef.
\item 
$\Nef_{BC}(X)=\Psef_{A}(X)^{*}$ and 
$\Psef_{BC}(X)\setminus\{0\}=\{\alpha \in H^{1,1}_{BC}(X,\R) \ | \ \langle\alpha,\eta\rangle>0 \ \forall \eta\in G\}$.
In particular
$\Psef_{BC}(X) = \Nef_{A}(X)^{*}$,
$\overline{\Psef_{A}(X)} = \Nef_{BC}(X)^{*}$ and
$\overline{\Nef_{A}(X)}=\Psef_{BC}(X)^{*}$,
\item
$G$ is open and $\overline{\Nef_{A}(X)}=\overline G$.
\item
If $E_{j}$ are the irreducible curves of negative self-intersection on $X$, then 
$$\Psef_{BC}(X)=\Nef_{BC}(X)+\sum_{j}\{[E_{j}]\}_{BC}$$
and
$$\overline{\Psef_{A}(X)}=\overline{\Nef_{A}(X)}+\sum_{j}\{[E_{j}]\}_{A}.$$
\item
$\overline{\Nef_{A}(X)}={\{\alpha\in H^{1,1}_{A}(X,\R) \ | \ \forall \epsilon>0 \ \exists \eta_{\epsilon}\in \alpha\cap\cE_{\R}^{1,1}(X) \ \eta_{\epsilon}>-\epsilon\omega\}}$.
\end{enumerate}
\end{Prop}
\begin{proof}
\begin{enumerate}
\item 
By arguing similarly to \cite[Section 2]{HarveyLawson} one gets the following facts: the operator $i\partial\bar\partial:\cE^{1,1}_{\R}(X)\to \cE^{2,2}_{\R}(X)$
has closed range since its cokernel is finite dimensional, \cite[Lemme 2]{SerreDualitate}, its dual 
$i\partial\bar\partial:{\cD'}^{0}_{\R}(X)\to{\cD'}^{1,1}_{\R}(X)$
 has closed range by the closed range theorem, \cite[IV 7.7]{Schaefer}, and thus the quotient topology induced by the projection $\pi:\{T\in{\cD'}_{\R}^{1,1}(X) \ | \ \d T=0\}\to H^{1,1}_{BC}(X,\R)$ on $H^{1,1}_{BC}(X,\R)$ is separated. Now the cone of closed positive currents is generated by the compact set $K:=\{ T\in {\cD'}_{\R}^{1,1}(X) \ | \ \d T=0, \ \langle T,\omega\rangle=1\}$, hence $\Psef_{BC}(X)$ is generated by its image $\pi(K)$ in $H^{1,1}_{BC}(X,\R)$. This image is compact and does not contain $0$. Thus $\Psef_{BC}(X)$ is closed in $H^{1,1}_{BC}(X,\R)$. The same argument shows that $\Nef_{BC}(X)$ is closed as well.
\item 
Let us denote the cone $\{\alpha\in H^{1,1}_{BC}(X,\R) \ | \ \forall \epsilon>0 \ \exists \eta_{\epsilon}\in \alpha\cap\cE_{\R}^{1,1}(X) \ \eta_{\epsilon}\ge-\epsilon\omega\}$ by $P_{nef}(X)$. 
The inclusion $\Nef_{BC}(X)\supset P_{nef}(X)$ is proved in 
\cite[Proposition 4.1]{LamariJMPA}. We show the second assertion by duality. 
Let $T$ be a closed nef current on $X$. 
By \cite[Th\'eor\`eme 1.2]{LamariJMPA} the Bott-Chern cohomology class $\{T\}_{BC} $ is in $P_{nef}(X)$ if for all positive pluriharmonic currents $T'$ one has $\langle\{T\}_{BC},\{T'\}_{A}\rangle\ge0$, where $\{T'\}_{A}$ is the Aeppli cohomology class of $T'$. 
By Proposition \ref{prop:decomposition} the positive, $i\partial\bar\partial$-closed $(1,1)$-current $T'$ has a decomposition $T'=\sum_jc_j[E_j]+T'',$ where $c_j\geq 0$ are positive real numbers, $E_{j}$ are irreducible compact curves on $X$ and $T''$ is a nef current. 
The inequality $\langle\{T\}_{BC},\{T''\}_{A}\rangle\ge0$,  is a consequence of Lemma \ref{lem:intersection} below. 
Now if we write $T$ as a limit of Gauduchon forms, $T=\lim_{n\to\infty}\omega_{n}$,  and choose a smooth representative $\eta$ in the class $\{E\}$ of the integration current along a curve $E$, we get 
$\langle\{T\}_{BC},\{E\}_{A}\rangle=\langle\{T\}_{BC},\{\eta\}_{A}\rangle = T(\eta)=\lim_{n\to\infty}\int_{X}\omega_{n}\wedge\eta= \lim_{n\to\infty}\int_{E}\omega_{n}\ge0$. 
\item
It is proved in \cite[Th\'eor\`eme 7.1]{LamariJMPA} that $P_{nef}(X)=\R_{\ge0}\{\tau\}_{BC}$, so $\Nef_{BC}(X)=P_{nef}(X)=\R_{\ge0}\{\tau\}_{BC}$.
\item
Let $T$ be a positive closed current with nef class $\{T\}_{BC}$. Then as before $T$ has a decomposition $T=\sum_jc_j[E_j]+T',$ and this time $T'$  is closed and nef.  Both $T$ and $T'$ are thus $\d$-exact. This implies that
 $\sum_jc_j[E_j]$ is $\d$-exact as well. If $X$ is elliptic the sum $\sum_jc_jE_j$ may be infinite but it  is in any case nef. If $X$ is not elliptic,  the divisor $\sum_jc_jE_j$ on  $X$ is homologically trivial and the corresponding integration current is nef.
\item
This follows from \cite[Th\'eor\`eme 1.2]{LamariJMPA}.
\item
As in (1) (see also 
\cite[Lemma 6]{HarveyLawson}) one can see that the operators $p_{1}\circ\d:\cE^{1}_{\R}(X)\to \cE^{1,1}_{\R}(X)$ and  $p_{2}\circ\d: {\cD'}^{1}(X)\to {\cD'}^{1,1}_{\R}(X)$ have closed range, where $p_{1}:\cE^{2}_{\R}(X)\to\cE^{1,1}_{\R}(X)$ and $p_{2}:{\cD'}^{1,1}_{\R}(X)\to{\cD'}^{1,1}_{\R}(X)$ are the natural projections. Thus  the quotient topologies induced on $H^{1,1}_{A}(X,\R)$ both from the space of pluriharmonic forms and from the space of pluriharmonic currents are separated. It follows that  $G$ is open and that if $T=\lim_{n\to\infty} \omega_{n}$ is a weak limit of Gauduchon forms, then $\{T\}_{A}=\lim_{n\to\infty}\{\omega_{n}\}_{A}\in \overline G$ hence $\overline{\Nef_{A}(X)}=\overline G$.
\item
This assertion is a consequence of \cite[Proposition 4.3]{LamariJMPA}. Note however that in loc. cit. one needs to take the closure of $\Psef_{A}(X)$, see also Remark \ref{rem:nonclosed}.
\item
We denote the set $\{\alpha\in H^{1,1}_{A}(X,\R) \ | \ \forall \epsilon>0 \ \exists \eta_{\epsilon}\in \alpha\cap\cE_{\R}^{1,1}(X) \ \eta_{\epsilon}>-\epsilon\omega\}$ by $\Pi_{nef}(X)$. Let $\alpha\in\Pi_{nef}(X)$ and let $\eta_{\epsilon}\in \alpha\cap\cE_{\R}^{1,1}(X)$ be such that $ \eta_{\epsilon}>-\epsilon\omega$. We set $\Omega_{\epsilon}:=\eta_{\epsilon}+\epsilon\omega$. Then the classes $\{\Omega_{\epsilon}\}_{A}=\alpha+\epsilon\{\omega\}_{A}$ are in $G$ and tend to $\alpha$ as $\epsilon$ tends to zero. Thus $\alpha\in\overline G= \overline{\Nef_{A}(X)}$ and $\Pi_{nef}(X)\subset \overline{\Nef_{A}(X)}$. Conversely, since we clearly have $G\subset \Pi_{nef}(X)$, we  get $\overline{\Nef_{A}(X)}\subset \overline {\Pi_{nef}(X)}$ and the desired equality of cones follows since $\Pi_{nef}(X)$ is closed, \cite[Lemma 2.3]{ChiRasSuv}.
\end{enumerate}
\end{proof}

\begin{Rem}\label{rem:nonclosed} The above proof cannot be mimicked to show closedness for $\Psef_{A}(X)$ and $\Nef_{A}(X)$ since there exist non-trivial $\d$-exact currents on $X$, hence the projection to $H^{1,1}_{A}(X,\R)$ of a corresponding generating compact  set of positive pluriharmonic currents will contain $0$. 

In fact, if $X$ is an Enoki surface with just one irreducible curve $C$, one can renormalize $\tau$ so that $\{\tau\}_{BC}=\{[C]\}_{BC}$ and one gets $\dim H^{1,1}_{BC}(X,\R)=\dim H^{1,1}_{A}(X,\R)=2 $,
$\Psef_{BC}(X)=\Nef_{BC}(X)=\R_{\ge0}\{\tau\}_{BC}$. 
By  Proposition \ref{prop:decomposition} it follows that any positive pluriharmonic current is nef. Moreover, if such a current vanishes on $\tau$, then it must be $\d$-exact by Proposition \ref{prop:pluriharmonic}. Hence we get
$$\Psef_{A}(X)=\Nef_{A}(X)=\{\alpha\in H^{1,1}_{A}(X,\R) \ | \ \langle\alpha,\{\tau\}_{BC}\rangle>0\}\cup\{0\}.$$ 
\end{Rem}

\begin{Lemma}\label{lem:intersection}
Let $T$, $T'$ be nef pluriharmonic $(1,1)$-currents on $X$ such that $T$ is $\d$-closed. Then for any sequences $(\omega_{n})_{n}$, $(\omega'_{n})_{n}$
of Gauduchon forms converging weakly to $T$ and to $T'$ respectively, we have:
$$\lim_{n,m\to\infty}\langle\omega_{n},\omega'_{m}\rangle=\langle\{T\}_{BC},\{T'\}_{A}\rangle.$$
\end{Lemma}
\begin{proof}
Let $\alpha_{1}, ..., \alpha_{n}$ be closed $(1,1)$-forms on $X$ whose classes generate $H^{1,1}_{dR}(X,\R)$ and such that $\int_{X}\alpha_{i}\wedge\alpha_{j}=-\delta_{ij}.$ Then in Aeppli cohomology $T'$ is cohomologous to some form $\delta\omega+A'$, where $A'=\sum_{j=1}^{n}a'_{j}\alpha_{j}$, $\delta, a_{j}\in\R$ and $\delta=\langle T',\tau\rangle\ge0$. Similarly $\omega'_{n}$ are cohomologous to some $(\delta+\epsilon'_{n})\omega+A'_{n}$, with $A'_{n}=\sum_{j=1}^{n}a'_{n,j}\alpha_{j}$. Evaluating on $\tau$ and on each $\alpha_{j}$ one obtains $\lim_{n\to\infty}\epsilon'_{n}=0$ and $\lim_{n\to\infty}a'_{n,j}=a'_{j}$. Thus 
$$\omega'_{n}=(\delta+\epsilon'_{n})\omega+A'_{n}+\bar\partial\sigma'_{n}+\partial\bar\sigma'_{n}$$ for some $(1,0)$-forms $\sigma'_{n}$. We have
$$0\ge\int_{X}(A'_{n})^{2}=\int_{X}  (A'_{n}+\d( \sigma'_{n}+\bar\sigma'_{n}))^{2}=  \int_{X}  (\omega'_{n}-(\delta+\epsilon'_{n})\omega+\partial \sigma'_{n}+\bar\partial\bar\sigma'_{n})^{2}=  $$
$$  \int_{X}  (\omega'_{n}-(\delta+\epsilon'_{n})\omega)^{2}+2\int_{X}\partial \sigma'_{n}\wedge\bar\partial\bar\sigma'_{n}= $$ 
$$\int_{X}  (\omega'_{n})^{2}-2\int_{X}(\delta+\epsilon'_{n})\omega\wedge \omega'_{n}+\int_{X}(\delta+\epsilon'_{n})^{2}\omega^{2} +2\parallel\partial \sigma'_{n}\parallel^{2}_{L^{2}},     $$
hence $\parallel\partial \sigma'_{n}\parallel^{2}_{L^{2}}\le \int_{X}(\delta+\epsilon'_{n})\omega\wedge \omega'_{n}$ and the right hand term tends to $\delta\langle T',\omega\rangle$ when $n$ tends to infinity. Thus the sequence $(\parallel\partial \sigma'_{n}\parallel_{L^{2}})_{n}$ is bounded. 

The same argument works for $T$ and this time we get 
$$\omega_{n}=\epsilon_{n}\omega+A_{n}+\bar\partial\sigma_{n}+\partial\bar\sigma_{n},$$
with $A_{n}=\sum_{j=1}^{n}a_{n,j}\alpha_{j}$,  $\lim_{n\to\infty}\epsilon_{n}=0$, $\lim_{n\to\infty}a_{n,j}=0$, and  $\lim_{n\to\infty}\|\partial\sigma'_{n}\|_{L^{2}}=0$.

Thus 
$$\lim_{n,m\to\infty}\langle\omega_{n},\omega'_{m}\rangle=\lim_{n,m\to\infty}\langle \epsilon_{n}\omega+A_{n}+\bar\partial\sigma_{n}+\partial\bar\sigma_{n}, (\delta+\epsilon'_{m})\omega+A'_{m}+\bar\partial\sigma'_{m}+\partial\bar\sigma'_{m}\rangle=$$
$$\lim_{n,m\to\infty}\langle \epsilon_{n}\omega+\bar\partial\sigma_{n}+\partial\bar\sigma_{n}, (\delta+\epsilon'_{m})\omega+\bar\partial\sigma'_{m}+\partial\bar\sigma'_{m}\rangle= $$
$$ \lim_{n,m\to\infty}\langle \epsilon_{n}\omega+\bar\partial\sigma_{n}+\partial\bar\sigma_{n}, (\delta+\epsilon'_{m})\omega\rangle+
\lim_{n,m\to\infty}\langle \epsilon_{n}\omega+\bar\partial\sigma_{n}+\partial\bar\sigma_{n}, \bar\partial\sigma'_{m}+\partial\bar\sigma'_{m}\rangle= $$
$$ \langle T,\delta\omega\rangle+\lim_{n,m\to\infty}\langle \epsilon_{n}\omega, \bar\partial\sigma'_{m}+\partial\bar\sigma'_{m}\rangle=\langle T,\delta\omega\rangle=\langle\{T\}_{BC},\{T'\}_{A}\rangle.$$
\end{proof}

%%%%%%%%%%%%%%%%%%%%%%%%%%%%%%%%%%%%%%%%%%%%%%%%%%%%%%%%%%%%%%%%%%%%%%%
%\bibliographystyle{apacite}

%\bibliography{../../../bib2020}
%\bibliographystyle{amsalpha}

%\begin{thebibliography}{BFM79}
%
%\bibitem[AG06]{AG}
%Vincenzo Ancona and Bernard Gaveau, \emph{Differential forms on singular
%  varieties}, Pure and Applied Mathematics (Boca Raton), vol. 273, Chapman \&
%  Hall/CRC, Boca Raton, FL, 2006, De Rham and Hodge theory simplified.
%
%
%\end{thebibliography}

 \hrule \medskip
\par\noindent

\end{document}